\documentclass[12pt]{article}
\usepackage{amsmath,amssymb}
\usepackage{amsthm}
\usepackage{graphicx,tikz}
\usepackage{enumerate}
\newtheorem{theorem}{Theorem}[section]
\newtheorem{proposition}[theorem]{Proposition}

\newtheorem{definition}[theorem]{Definition}
\newtheorem{remark}[theorem]{Remark}
\newtheorem{example}[theorem]{Example}
\newtheorem{lemma}[theorem]{Lemma}
\newtheorem{cor}[theorem]{Corollary}
\newtheorem{pro}[theorem]{Proposition}
\numberwithin{equation}{section}
\begin{document}
\title{
Tempered reductive homogeneous spaces
}

\author{Yves Benoist and Toshiyuki Kobayashi
}
\date{}
\maketitle






\centerline{ 
}
\begin{abstract}{
Let $G$ be a semisimple algebraic Lie group and $H$ a reductive subgroup.
We find geometrically the best even integer $p$ for which the  representation  of $G$
in $L^2(G/H)$ is almost $L^{p}$. As an application,
we give a criterion which detects whether this  representation is tempered.
}\end{abstract}

\maketitle

\noindent
\textit{Key words and phrases,} Lie groups, homogeneous spaces, tempered 
representations, matrix coefficients, symmetric spaces


\section{Introduction}
\label{sec:intro}


Let $G$ be an algebraic semisimple Lie group and 
$H$  a reductive subgroup.
The natural unitary representation of $G$ in $L^2(G/H)$ 
has been studied over years since 
 the pioneering work of I.~M.~Gelfand and Harish-Chandra. 

Thanks to many mathematicians including
E.~van den Ban, 
P.~Delorme, M.~Flensted-Jensen, S.~Helgason, T.~Matsuki, T.~Oshima, H.~Schlichtkrull, J. Sekiguchi, among others,
many properties of this representation are known when 
$G/H$ is a symmetric space, i.e.~when $H$ 
is the set of fixed points of an involution of $G$.
Most of the preceeding works in this case are built on the fact that the ring $\mathbb D(G/H)$ of $G$-invariant differential operators is commutative, and that the disintegration of $L^2(G/H)$ (Plancherel formula) is essentially the expansion of $L^2$-functions into joint eigenfunctions of $\mathbb D(G/H)$.

This paper deals with a more general reductive subgroup $H$,
for which we cannot expect that the ring $\mathbb D(G/H)$ is commutative,
and a complete change of the machinery would be required in the study of $L^2(G/H)$.
We address the following question: 
\textsl{What kind of unitary representations occur
in the disintegration of $G/H$?} More precisely, 
\textsl{when are all of them tempered?}

The aim of this paper is to give an easy-to-check
necessary and sufficient condition  
on $G/H$   under which all these irreducible unitary representations are 
tempered, or equivalently under which $L^2(G/H)$ is tempered,
 and in particular, has a \lq uniform spectral gap\rq.
We note that irreducible tempered representations
 were completely classified more than 30 years ago by
 Knapp and Zuckerman in \cite{KnZu82},
whereas non-tempered ones are still
mysterious and have not been completely understood.
Our criterion singles out homogeneous spaces $G/H$ for which irreducible
 non-tempered
 unitary representations occur in the disintegration of $L^2(G/H)$.
More generally, 
we give, for any even integer $p$, 
a necessary and sufficient condition 
under which  $L^2(G/H)$ is almost $L^{p}$
(see Theorem \ref{thm:Xtemp}).

Our criterion is new even when $G/H$ is a reductive symmetric space where the disintegration of $L^2(G/H)$  was established up to the classification of discrete series representations for (sub)symmetric spaces 
(\cite{BS, De, O}). 
Indeed irreducible unitary representations that contribute to $L^2(G/H)$
in the direct integral are obtained as a parabolic induction from
discrete series for subsymmetric spaces, but a subtle point arises from discrete series  with singular parameter. 
In fact, all possible discrete series were captured in \cite{OM},
 however,
  the non-vanishing conditions of these modules
 are sometimes combinatorially complicated
and
 these very modules with singular parameters
would affect the worst decay of
 matrix coefficients if they do not vanish.
(Such complication does not occur in the case of group manifolds
 because Harish-Chandra's discrete series do not allow singular parameters.)
Algebraically, the underlying $(\mathfrak g, K)$-modules are certain
Zuckerman derived functor modules $A_\mathfrak q(\lambda)$
(see \cite{Knapp-Vogan} for general theory) 
with possibly singular $\lambda$  crossing many walls of the Weyl chambers,
so that the Langlands parameter may behave in an unstable way and 
even the modules themselves may disappear.
A necessary condition for the non-vanishing of discrete series
for reductive symmetric spaces with singular
parameter was proved in \cite{Ma} that corrected
the announcement in \cite{OM}, whereas
a number of general methods to verify the 
non-vanishing of $A_{\mathfrak q}(\lambda)$-modules
 have been developed more recently in \cite[Chapters 4, 5]{Kupq}, \cite{T} 
for some classical groups,
but the proof of the sufficiency of the non-vanishing condition in \cite{Ma}
 has not been given so far.

Beyond symmetric spaces, 
very little has been known 
on the unitary representation of $G$ in $L^2(G/H)$ (cf.~\cite{Kdisc}).

\vskip 1pc

Here is an outline of the paper.
As a baby case, we  first study the 
unitary representations of a semisimple group in $L^2(V)$ where
$V$ is a finite dimensional representation.
We give a necessary and sufficient condition on $V$   
under which the representation  in $L^2(V)$ is tempered
(Theorem \ref{thm:Vtemp}), or more generally,
under which this representation  is almost $L^p$.
The heart of the paper is Chapter 4 where we give a proof of the main results
(Theorem \ref{thm:Xtemp})  for reductive homogeneous spaces $G/H$. 
In a subsequent paper we see that this criterion suffices to
give a  complete classification of 
 the pairs $(G, H)$ of algebraic reductive groups
for which the unitary representation of $G$ on $L^2(G/H)$ is non-tempered.
To give a flavor of what is possible, we collect a few applications of this criterion in Chapter 5, omitting the details of the computational verification. 

\section{Preliminary results}
\label{sec:tempered}

We collect in this chapter a few well-known facts on almost $L^p$ represen\-ta\-tions, on tempered representation 
and on uniform decay of matrix coefficients. 

\subsection{Almost $L^p$ representations}
\label{subsec:wlpr}

In this paper all Lie groups will be real Lie groups.
Let $G$ be a unimodular Lie group and $\pi$ be
a unitary representation  of $G$ in a Hilbert space ${\mathcal H}_\pi $.

\begin{definition}\label{def:wlp}
Let $p\geq 2$. The unitary representation $\pi$ is said to be almost $L^p$
if there exists  a dense subset $D\subset{\mathcal H}_\pi $ 
for which the coefficients\\
 $c_{v_1,v_2} \colon g\mapsto \langle  \pi(g)v_1,v_2\rangle $ are in
$L^{p+\varepsilon}(G)$ for all $\varepsilon>0$ and all $v_1$, $v_2$ in $D$. 
\end{definition}

Let $K$ be a maximal 
compact subgroup of $G$.

\begin{lemma}
A unitary representation $\pi$ is almost $L^p$
if and only if there exists  a dense subset $D_0\subset{\mathcal H}_\pi $ of 
$K$-finite vectors
for which the coefficients
 $c_{v_1,v_2}$ are in
$L^{p+\varepsilon}(G)$ for all $\varepsilon>0$ and all $v_1$, $v_2$ in $D_0$. 
\end{lemma}

\begin{proof}
We first notice that for all $v_1$, $v_2$ in $D$ and all $k_1$, $k_2$ in $K$
the two vectors $\pi(k_1)v_1$ and $\pi(k_2)v_2$ have a coefficient with same $L^{p+\varepsilon}$-norm:
$$\| c_{\pi(k_1)v_1,\pi(k_2)v_2}\|_{L^{p+\varepsilon}}= \| c_{v_1,v_2}\|_{L^{p+\varepsilon}}.$$
Let $dk$ be the Haar probability measure on $K$.
For any two $K$-finite functions $f_1$ and $f_2$ on $K$,
bounded by $1$, the two vectors $w_1:=\int_Kf_1(k)\pi(k)v_1 \, dk$ and 
$w_2:=\int_Kf_2(k)\pi(k)v_2 \, dk$ have a coefficient with bounded $L^{p+\varepsilon}$-norm:
$$\| c_{w_1,w_2}\|_{L^{p+\varepsilon}}\leq \| c_{v_1,v_2}\|_{L^{p+\varepsilon}}.$$
These vectors $w_i$ live in a dense set $D_0$ of $K$-finite vectors of ${\mathcal H}_\pi $.
\end{proof}

\subsection{Tempered representations}
\label{subsec:temr}

The following definition is due to Harish-Chandra 
(See also \cite[Appendix F]{BeHaVa})

\begin{definition}
The unitary representation $\pi$ is said to be tempered
if $\pi$ is weakly contained 
in the regular representation $\lambda_G$ of $G$ in $L^2(G)$ i.e. if every coefficient of $\pi$ 
is a uniform limit on every compact of $G$  of a sequence of 
sums of coefficients of $\lambda_G$. 
\end{definition}

Here are a few basic facts on tempered representations.\\
- Let $G'\subset G$ be a finite index subgroup.
A unitary representation $\pi$ of  $G$ is tempered if and only if $\pi$
is tempered as a representation of $G'$.\\
- A unitary representation $\pi$ of a reductive group $G$ is tempered 
if and only if $\pi$ is tempered 
as a representation of the derived subgroup $[G,G]$.\\

\begin{proposition} \label{pro:temwl2}
{\bf (Cowling, Haagerup, Howe)}
Let $G$ be a semisimple connected  Lie 
group with finite center, and $m$ a positive integer. 

A unitary representation $\pi$ of $G$ is almost $L^2$
if and only if $\pi$ is tempered.  

More generally,  $\pi$  is almost $L^{2m}$
if and only if $\pi^{\otimes m}$ is tempered.  
\end{proposition}

See \cite[Theorems 1, 2 and Corollary]{CHH}.

\begin{remark}
When $G$ is amenable, according to Hulanicki--Reiter Theorem 
(see \cite[Theorem G.3.2]{BeHaVa}), every unitary representation of $G$ is tempered. 
However, when $G$ is non-compact, the trivial representation is not almost $L^2$.  
\end{remark}

The following remark was used implicitly in the introduction.

\begin{remark}
When a unitary representation $\pi$ of $G$ is a direct integral 
$\pi=\int^\oplus\pi_\lambda\,{\rm d}\mu(\lambda)$ of irreducible unitary representations $\pi_\lambda$, {\it the representation $\pi$ 
is tempered if and only if the representations $\pi_\lambda$ 
are tempered for $\mu$-almost every parameter $\lambda$}.
\end{remark}

\begin{proof}
Indeed, $\pi$ is weakly contained in the direct sum representation $\oplus_\lambda\pi_\lambda$, 
and conversely $\pi_\lambda$ is weakly contained in $\pi$  for $\mu$-almost every $\lambda$.

These statements follow for instance from the following fact 
in \cite[Theorem F.4.4]{BeHaVa} or \cite[Section 18]{Dix69}~:
For two unitary representations 
$\rho$ and $\rho'$ of $G$, one has the equivalence:\\
\centerline{\it $\rho$ is weakly contained in $\rho'$ $\Longleftrightarrow$
$\|\rho(f)\|\leq\|\rho'(f)\|$ for all $f$ in $L^1(G)$,}
where $\rho(f)=\int_Gf(g)\rho(g){\rm d} g$.
Note that this condition has only to be checked for a countable dense 
set of functions $f$ in $L^1(G)$, and that one has the equality 
$\|\pi(f)\|={\rm supess}_\lambda\|\pi_\lambda(f)\|$ (see 
\cite[Section II.2.3]{Dix57}).\end{proof}

\subsection{Uniform decay of coefficients}
\label{subsec:udc}

Let $G$ be a linear semisimple  connected  Lie group
and let $\Xi$ be the Harish-Chandra spherical function on $G$ 
(see \cite{CHH}).
A short definition for $\Xi$ is as the coefficient of the normalized $K$-invariant vector
of the spherical representation of the unitary principal series 
$\pi_o={\rm Ind}_P^G({\bf 1}_P)$ 
where $P$ is a minimal parabolic subgroup of $G$.
In this paper we will not need the precise formula for $\Xi$ 
but just the fact that
this function $\Xi$ is in $L^{2+\varepsilon}(G)$ for all $\varepsilon>0$
and the following proposition.


\begin{proposition}\label{pro:wlphc}
{\bf (Cowling, Haagerup, Howe )}
Let $p$ be an even integer.
A unitary representation $\pi$ of $G$  is almost $L^p$ if and only if,
for every $K$-finite vectors $v$, $w$ in ${\mathcal H}_\pi $, for every $g$ in $G$,
one has 
$$
|\langle  \pi(g)v,w\rangle |\leq \Xi(g)^{2/p} 
\| v\| \| w\| (\dim \langle  K v\rangle )^\frac12 (\dim \langle  K w\rangle )^\frac12.
$$
\end{proposition}

See \cite[Corollary p.108]{CHH}

This proposition tells us that once an almost $L^p$-norm condition is checked for 
the coefficients of a dense 
set of vectors of ${\mathcal H}_\pi $, 
one gets a {\sc uniform} estimate for the coefficients of {\sc all} the $K$-finite
vectors of ${\mathcal H}_\pi $. 

In this proposition, the assumption that the real number $p\geq 2$ is an even integer
can probably be dropped. If this is the case, the same assumption can also be dropped
in our Theorems \ref{thm:Vtemp} and \ref{thm:Xtemp}. 
\vspace{1em}

The set of $p$ for which $\pi$ is almost $L^p$ is an interval
$[p_\pi,\infty[$ with $p_\pi\geq 2$ or $p_\pi=\infty$. 
Even though we will not use them, we recall the following  
two important properties of these constant $p_\pi$. 

When $G$ is quasisimple of higher rank and ${\mathcal H}_\pi $ does not contain $G$-invariant
vectors, this real number $p_\pi$ is bounded by a  constant
$p_G<\infty$ (see \cite{Oh1}). 

According to Harish-Chandra, when $G$ is semisimple and $\pi$ is irreducible with finite kernel, 
this real $p_\pi$ is finite (see \cite[Theorem 8.48]{Kn}). 

\subsection{Representations in $L^2(X)$}
\label{subsec:rrx}

Let $X$ be a locally compact space endowed with a continuous action of $G$ 
preserving a Radon measure $\operatorname{vol}$ on $X$.
One has a natural representation $\pi$ of $G$ in $L^2(X)$ given by,
$(\pi(g)\varphi)(x)=\varphi(g^{-1}x)$ for $g$ in $G$, $\varphi$ in $L^2(X)$ and $x$ in $X$.

\begin{lemma}\label{lem:wlpVX}
Let $G$ be a semisimple linear connected  Lie 
group, $p$  a positive even integer,
and $X$ a locally compact space endowed with a continuous action of $G$ 
preserving a Radon measure $\operatorname{vol}$.

The representation of $G$ in $L^2(X)$  is almost $L^{p}$ if and only if,
for any compact subset 
$C$ of $X$ and any $\varepsilon>0$ 
$\operatorname{vol} (gC\cap C)\in L^{p+\varepsilon}(G)$.
\end{lemma}

\begin{proof}
If the representation of $G$ in $L^2(X)$  is almost $L^{p}$ then, according to Proposition
\ref{pro:wlphc}, for all $K$-invariant compact set $B$ of $X$,
the function 
$g\mapsto \operatorname{vol} (gB\cap B)=
\langle \pi(g){\bf 1}_B,{\bf 1}_B\rangle $
belongs to $ L^{p+\varepsilon}(G)$.
Since any compact set $C$ of $X$ is included in such 
a $K$-invariant compact set $B$, the function
$g\mapsto \operatorname{vol} (gC\cap C)$
belongs also to $ L^{p+\varepsilon}(G)$.

Conversely, let $D\subset L^2(X)$ be the dense subspace 
of continuous compactly supported 
functions on $X$.
For any two continuous functions $\varphi_1, \varphi_2\in D$, 
the coefficient $\langle \pi(g)\varphi_1,\varphi_2\rangle $
is bounded by 
$\|\varphi_1\|_\infty\|\varphi_2\|_\infty \operatorname{vol} (gC\cap C)$ 
where $C:={\rm Supp}(\varphi_1)\cup{\rm Supp}(\varphi_2)$, 
and hence this coefficient belongs to $ L^{p+\varepsilon}(G)$.
\end{proof}

\section{Representations in $L^2(V)$}
\label{sec:linear}

In this chapter we study the representation of a semisimple Lie group
in $L^2(V)$ where $V$ is a finite dimensional representation.

\subsection{Function $\rho_V$}
\label{subsec:rhoV}

Let $H$ be a  reductive algebraic Lie group,
and $\tau \colon H \to SL_\pm(V)$ a finite dimensional algebraic
representation over $\mathbb{R}$ preserving the Lebesgue measure on $V$.
We  write $d\tau \colon \mathfrak{h} \to \operatorname{End}(V)$ for the differential
representation of $\tau$. Let $\mathfrak{a}=\mathfrak{a}_\mathfrak{h}$ be a 
maximal split abelian subspace in $\mathfrak{h}$.

For an element $Y$ in $\mathfrak{a}$,
we denote by $V_+$ the sum of eigenspaces of $\tau(Y)$ having
positive eigenvalues, 
and set
\begin{equation}\label{eqn:rhoV}
\rho_V(Y)
: = \operatorname{Trace}_{V_+} (d\tau(Y)).
\end{equation}
Since this function $\rho_V \colon \mathfrak{a} \to \mathbb{R}_{\ge0}$ 
will be very important in our analysis,
we begin by a few trivial but useful comments. 
We notice first that, since $H$ is volume preserving, for any $Y\in \mathfrak{a}$,
\begin{equation}\label{eqn:rhopm}
\rho_V(-Y) = \rho_V(Y), 
\end{equation}
\begin{equation}\label{eqn:rho0}
\rho_V(Y) = 0 \Leftrightarrow d\tau(Y)=0.
\end{equation}
This function $\rho_V$ is  invariant under the
finite group 
$W_H := N_H(\mathfrak{a})/Z_H(\mathfrak{a})$.
This group is isomorphic to the Weyl group of the restricted root system
$\Sigma(\mathfrak{h},\mathfrak{a})$ if $H$ is connected.
This function $\rho_V$ is 
continuous and is piecewise linear
i.e. there exist finitely many convex polyhedral cones 
which cover $\mathfrak{a}$ and on which $\rho_V$ is linear.

\begin{example}
For $(\tau,V) = (\operatorname{Ad}, \mathfrak{h})$,
$\rho_{\mathfrak{h}}$ coincides with twice the usual `$\rho$' on the
positive Weyl chamber $\mathfrak{a}_+$ with respect to a positive
system $\Sigma^+(\mathfrak{h},\mathfrak{a})$.
\[
\rho_{\mathfrak{h}}
=  \sum_{\alpha\in\Sigma^{^+}\! (\mathfrak{h},\mathfrak{a})} 
\dim \mathfrak{h}_\alpha \,\alpha
\quad\text{on\/ $\mathfrak{a}_+$},
\]
where $\mathfrak{h}_\alpha\subset \mathfrak{h}$ is the root subspace 
associated to $\alpha$.
\end{example}

For other representations $(\tau,V)$, the maximal convex polyhedral cones on which 
$\rho_V$ is linear are most often much smaller than the Weyl chambers.

\subsection{Criterion for temperedness of $L^2(V)$}

Since the Lebesgue measure on $V$ is
$H$-invariant,
we have a natural unitary representation of $H$ on $L^2(V)$
as in Section \ref{subsec:rrx}.

\begin{theorem}\label{thm:Vtemp}
Let $H$  an algebraic  semisimple Lie group,
$\tau \colon H \to SL_\pm(V)$ an  algebraic
representation and $p$ a positive even integer.
Then, one has the equivalences~:\\
$a)$
$L^2(V)$ is tempered \;\;
$\Longleftrightarrow$\;
$\rho_{\mathfrak{h}}(Y) \le 2\, \rho_V(Y)$ for any $Y\in\mathfrak{a}$.\\
$b)$ 
$L^2(V)$ is almost $L^p$ \;
$\Longleftrightarrow$\;
$\rho_{\mathfrak{h}}(Y) \le p\, \rho_V(Y)$ for any $Y\in\mathfrak{a}$.
\end{theorem}

\begin{remark} 
Inequality
$\rho_{\mathfrak{h}}\le p\, \rho_V$ holds on
$\mathfrak{a}$ 
if and only if it holds
on  $\mathfrak{a}_+$.

Since all the maximal split abelian subspaces of $\mathfrak{h}$ are $H$-conjugate, 
it is clear that this condition
does not depend on the choice of $\mathfrak{a}$.
\end{remark}

\begin{example}
Let $H={\rm SL}(2,\mathbb{R})^d$ with $d\geq 1$.
The unitary representation in $L^2(V)$ is  tempered if and only if 
the kernel of $\tau$ is finite.
\end{example}

\begin{example}
Let $H={\rm SL}(3,\mathbb{R})$.
The unitary representation in $L^2(V)$ is tempered if and only if 
$\dim (V/V^H) > 3$ where $V^H=\{ v\in V: Hv=v\}$.
\end{example}

For $h\in H$, $x\in V$  and a measurable subset $C \subset V$,
we write  $hx$ for $\tau(h)x$ and we set $hC: =\{hx\in V: x\in C\}$.
Similarly, for $a>0$ we set $aC: =\{ax\in V: x\in C\}$. 
We write $\operatorname{vol}(C)$ for the volume of $C$ with respect to
the Lebesgue measure.

\begin{proof}[Proof of Theorem \ref{thm:Vtemp}]
When the kernel of $\tau$ is noncompact, 
both  sides of the equivalence  are false.
Hence we may assume that the kernel of $\tau$ is compact.
Since $H$ is semisimple,
according to Proposition \ref{pro:temwl2} and
Lemma \ref{lem:wlpVX}, 
it is sufficient to prove the
following equivalence:
$$
\begin{minipage}{8em}\it
$\rho_{\mathfrak{h}}(Y)\le p\,\rho_V(Y)$ for any $Y\in\mathfrak{a}$
\end{minipage}\quad
\Longleftrightarrow\quad
\begin{minipage}{17em}\it
$\operatorname{vol}(hC\cap C)\in L^{p+\varepsilon}(H)$
for any compact subset $C$ in $V$ and any $\varepsilon>0$.
\end{minipage}
$$
This statement is a special case of Proposition \ref{prop:LpV} below.
\end{proof}

\subsection{$L^p$-norm of $\operatorname{vol}(hC\cap C)$}
\label{subsec:pplpv}

Suppose now that the kernel of $\tau$ is compact.
According to \eqref{eqn:rho0}, one has
$\rho_{V}(Y)>0$ as soon as $Y\ne 0$.
Hence the real number
\begin{equation}\label{eqn:pV}
p_{V}:=\max_{Y\in\mathfrak{a}\setminus\{0\}}
   \frac{\rho_{\mathfrak{h}}(Y)}{\rho_{V}(Y)}\, 
\end{equation}
is finite.

\begin{proposition}\label{prop:LpV} 
Let $H$ be an algebraic  reductive Lie group,
and $\tau \colon H \to SL_\pm(V)$  a volume preserving algebraic 
representation with compact kernel.
For any real $p>0$, one has the equivalence~:\\ 
\centerline{$p>p_{V}$ $\Longleftrightarrow$
$\operatorname{vol}(hC\cap C) \in L^p(H)$
for any compact set $C$ in $V$.}
\end{proposition}

In this section we will show how to deduce Proposition \ref{prop:LpV}
from a volume estimate that we will prove in the next section.

\begin{proof}[Proof of Proposition \protect\ref{prop:LpV}]
Let $H_K$ be a maximal compact subgroup of $H$
such that $H=H_K(\exp\mathfrak{a})H_K$ is a Cartan decomposition of $H$.

The Haar measure $dh$ of $H$ is given as
\begin{eqnarray}\label{eqn:haar}
\int_H f(h)dh
= \int_{\mathfrak{a}} f(e^Y) D_{\mathfrak{h}}(Y) dY
\end{eqnarray}
for any $H_K$-biinvariant measurable function $f$ on $H$,
where
\[
D_{\mathfrak{h}}(Y)
:= \prod_{\alpha\in\Sigma^{^+}\! (\mathfrak{h},\mathfrak{a})}
   |\sinh\langle\alpha,Y\rangle|^{\dim\mathfrak{h}_\alpha}
\ \ \text{for $Y\in\mathfrak{a}$}.
\]
We also introduce the function on $\mathfrak{a}$
$$
\widetilde{D}_{\mathfrak{h}}(Y) :=\int_{\|Z\|\leq 1}D_{\mathfrak{h}}(Y+Z)dZ.
$$

We shall prove successively  the following equivalences
\begin{eqnarray*} 
&(i)& \text{ $\operatorname{vol}(hC\cap C)\in L^p(H)$, for any  compact $C\subset V$,}\\
\Longleftrightarrow&(ii)&
\text{$\operatorname{vol}(e^YC\cap C)^p\, D_{\mathfrak{h}}(Y)\in 
L^1(\mathfrak{a})$, for any  compact $C\subset V$,}\\
\Longleftrightarrow&(iii)&
\text{$\operatorname{vol}(e^YC\cap C)^p\, \widetilde{D}_{\mathfrak{h}}(Y)\in 
L^1(\mathfrak{a})$, for any  compact $C\subset V$,}\\
\Longleftrightarrow&(iv)&
\text{$\operatorname{vol}(e^YC\cap C)^p\,  e^{\rho_{\mathfrak{h}}(Y)}\in 
L^1(\mathfrak{a})$, for any  compact $C\subset V$,}\\
\Longleftrightarrow&(v)&
\text{$e^{\rho_{\mathfrak{h}}(Y)-p\,\rho_V(Y)}\in 
L^1(\mathfrak{a})$,}\\
\Longleftrightarrow&(vi)&
\text{$p\,\rho_V(Y)-\rho_{\mathfrak{h}}(Y)>0$, for any $Y\in \mathfrak{a}\smallsetminus 0$.}
\end{eqnarray*}
$(i)\Longleftrightarrow (ii)$ We may choose $C$ to be  $H_K$-invariant by
expanding $C$ if necessary. We apply then
the integration formula  \eqref{eqn:haar} to
the $H_K$-biinvariant function
$\operatorname{vol}(hC\cap C)$.\\
$(ii)\Longleftrightarrow (iii)$ Replace $C$ by a larger compact $C':=e^{\mathfrak{a}(1)}C$
where $\mathfrak{a}(1)$ is the unit ball $\{ Z\in \mathfrak{a}\mid \| Z\|\leq 1\}$.
Since $\operatorname{vol}(e^{Y-Z} C \cap C) \leq
\operatorname{vol}(e^Y C' \cap C')$
for any $Z \in \mathfrak{a}(1)$,
one has, by using the change of variables $Y':=Y-Z$,
\begin{align*}
\int_{\mathfrak{a}}{\operatorname{vol}} (e^Y C \cap C)^p \widetilde{D}_{\mathfrak{h}}(Y)dY
&= \int_{\mathfrak{a}} \int_{\|Z\|\leq 1}
     \operatorname{vol}(e^Y C \cap C)^p D_{\mathfrak{h}}(Y+Z)dYdZ
\\
&\leq \operatorname{vol} (\mathfrak{a}(1))
   \int_{\mathfrak{a}} \operatorname{vol} (e^Y C'\cap C')^p D_{\mathfrak{h}} (Y)dY,
\\
\int_{\mathfrak{a}}{\operatorname{vol}} (e^{Y'} C \cap C)^p D_{\mathfrak{h}}(Y') dY'
&\leq \int_{\mathfrak{a}}\int_{\|Z\|\leq 1}{\operatorname{vol}} (e^{Y-Z} C' \cap C')^p D_{\mathfrak{h}} (Y)dYdZ
\\
&=\operatorname{vol}(\mathfrak{a}(1))^{-1}
    \int_{\mathfrak{a}}{\operatorname{vol}} (e^{Y'} C' \cap C')^p \widetilde{D}_{\mathfrak{h}} (Y') dY'.
\end{align*}
\\
$(iii)\Longleftrightarrow (iv)$ We notice that 
we can find constants $a_1,a_2>0$, such that for any $Y\in \mathfrak{a}$,
the following inequality holds: 
\[
a_1 
\,e^{\rho_{\mathfrak{h}}(Y)}
\le \widetilde{D}_{\mathfrak{h}}(Y)\le  a_2 \,e^{\rho_{\mathfrak{h}}(Y)}.
\]
$(iv)\Longleftrightarrow (v)$ We use Proposition \ref{prop:volV}, that we will prove 
in  Section \ref{subsec:bddC}, 
and that gives, for $C$ large enough,  constants 
$m, M>0$ such that, for any 
$Y\in \mathfrak{a}$,
\[
m \, e^{-\,\rho_V(Y)}
\le \operatorname{vol}(e^YC\cap C)
\le M \, e^{-\,\rho_V(Y)}.
\]
$(v)\Longleftrightarrow (vi)$ We recall that the function $\rho_{\mathfrak{h}}\! -\! p\rho_V$ is 
continuous and piecewise linear.

This proves Proposition \ref{prop:LpV} provided the following Proposition \ref{prop:volV}.
\end{proof}

\subsection{Estimate of $\operatorname{vol}(e^YC\cap C)$}
\label{subsec:bddC}
The following asymptotic estimate of
$\operatorname{vol}(e^Y C \cap C)$ for the linear representation
in $V$
will become a prototype of the
volume estimate for the action on $G/H$ 
which we shall discuss in Section \ref{sec:l2gh}
(Theorem \ref{thm:volX}).

\begin{proposition}\label{prop:volV}
Let $H$ be an algebraic  reductive Lie group,
$\tau\colon H \to SL_\pm(V)$ a volume preserving  algebraic 
representation, and
$C$ be a compact neigh\-borhood of $0$ in $V$.
Then there exist constants $m\equiv m_C>0$ , $M\equiv M_C > 0$ such that
\[
m e^{-\rho_V(Y)}
\le \operatorname{vol}(e^Y C \cap C)
\le M e^{-\rho_V(Y)}
\quad\text{for any $Y \in \mathfrak{a}$. 
}
\]
\end{proposition}

To see this, 
write
$\Delta \equiv \Delta(V,\mathfrak{a})\subset \mathfrak{a}^*$
for the set of weights of the representation
$d\tau|_{\mathfrak{a}}\colon \mathfrak{a} \to \operatorname{End}(V)$, 
and
\begin{equation}\label{eqn:Vwt}
\textstyle
V = \bigoplus_{\lambda\in\Delta} V_\lambda,
\quad
v = \sum v_\lambda
\end{equation}
for the corresponding weight space decomposition.

\begin{lemma}\label{lem:volV}
For each $\lambda\in\Delta$,
let  $B_\lambda$ be a convex neighborhood of $0$ in $V_\lambda$,
and let
$B := \prod_\lambda B_\lambda$.
Then, one has 
$$
\operatorname{vol}(e^Y B \cap B) 
 = \operatorname{vol}(B) e^{-\rho_V(Y)}
\quad \text{ for any $Y \in \mathfrak{a}$.}
$$
\end{lemma}

\begin{proof}
For any real $t$, one  has
$B_\lambda \cap e^{-t}B_\lambda = e^{-t^+}B_\lambda$ where $t^+:=\max(t,0)$. 
Then we get
\begin{equation*}\textstyle
B \cap e^{-Y} B
= \prod_\lambda (B_\lambda \cap e^{-\lambda(Y)} B_\lambda)
= \prod_\lambda e^{-\lambda(Y)^+} B_\lambda,
\end{equation*}
and $\quad\quad\quad\quad
\operatorname{vol}(e^Y B \cap B)
= \operatorname{vol}( B \cap e^{-Y}B)
= e^{-\rho_V(Y)} \operatorname{vol}(B).
$
\end{proof}

\begin{proof}[Proof of Proposition \ref{prop:volV}]
We take $\{B_\lambda\}$ and $\{B'_\lambda\}$ such that
\[\textstyle
\prod_\lambda B_\lambda \subset C \subset \prod_\lambda B'_\lambda
\]
and we apply Lemma \ref{lem:volV}.
\end{proof}

\section{Representations in $L^2(G/H)$}
\label{sec:l2gh}

In this chapter we study the representations of an algebraic semisimple Lie group
in $L^2(X)$ where $X$ is a homogeneous space with reductive isotropy.

\subsection{Criterion for temperedness of $L^2(G/H)$}

Let $G$ be an algebraic reductive Lie group
and $H$ an algebraic reductive subgroup of $G$.
Since the homogeneous space $X=G/H$ carries a $G$-invariant Radon measure,
there is a natural  unitary representation of $G$ on $L^2(G/H)$ 
as in Section \ref{subsec:rrx}.
We want to study the temperedness of this representation.

Let $\mathfrak{q}$ be an $H$-invariant complementary subspace  of
the Lie algebra $\mathfrak{h}$ of $H$ in $\mathfrak{g}$.
We fix a 
maximal split abelian subspace
$\mathfrak{a}$   of $\mathfrak{h}$ and we define
$\rho_{\mathfrak{q}}\colon \mathfrak{a} \to \mathbb{R}_{\ge0}$
for the $H$-module $\mathfrak{q}$
as in Section \ref{subsec:rhoV}.

Here is the main result of this chapter~:

\begin{theorem}\label{thm:Xtemp}
Let $G$ be an algebraic semisimple Lie group,
$H$ an algebraic reductive subgroup of $G$,
and $p$ a positive even integer.
Then, one has the equivalences~:\\
$a)$
$L^2(G/H)$ is tempered\;\;\;
$\Longleftrightarrow$\;
$\rho_{\mathfrak{h}}(Y) \le \rho_{\mathfrak{q}}(Y)$ for any $Y\in\mathfrak{a}$.\\
$b)$ 
$L^2(G/H)$ is almost $L^p$ \;
$\Longleftrightarrow$\;
$\rho_{\mathfrak{g}}(Y) \le p\,\rho_{\mathfrak{q}}(Y)$ for any $Y\in\mathfrak{a}$.
\end{theorem}

\begin{remark} 
Since $\rho_{\mathfrak{g}}=\rho_{\mathfrak{h}}+\rho_{\mathfrak{q}}$, 
one has the equivalence\\ 
\centerline{$\rho_{\mathfrak{h}} \le \rho_{\mathfrak{q}}
\Leftrightarrow \rho_{\mathfrak{g}} \le 2\,\rho_{\mathfrak{q}}$.}

The inequality
$\rho_{\mathfrak{g}}\le p\,\rho_{\mathfrak{q}}$ holds on
$\mathfrak{a}$ if and only if it holds
on $\mathfrak{a}_+$.
\end{remark}

\begin{proof}[Proof of Theorem \ref{thm:Xtemp}]
When the kernel of  the action of $G$ on $G/H$ is noncompact, 
both sides of the equivalence are false. 
Hence we may assume that this  kernel is compact.
But then,
according to 
Proposition \ref{pro:temwl2} and Lemma \ref{lem:wlpVX}, it is sufficient to prove the
following equivalence:
$$
\begin{minipage}{8em}\it
$\rho_{\mathfrak{g}}(Y)\le p\,\rho_{\mathfrak{q}}(Y)$ for any $Y\in\mathfrak{a}$
\end{minipage}\quad
\Longleftrightarrow\quad
\begin{minipage}{17em}\it
$\operatorname{vol}(gC\cap C)\in L^{p+\varepsilon}(G)$
for any compact subset $C$ in $G/H$ and any $\varepsilon>0$.
\end{minipage}
$$
This statement is a special case of Proposition \ref{prop:LpX} below.
\end{proof}

\subsection{$L^p$-norm of  $\operatorname{vol}(gC\cap C)$}

We assume that  the action of $G$ on $G/H$ has compact kernel
or, equivalently, that
the action of $H$ on $\mathfrak{q}$
has compact kernel. Then, according to \eqref{eqn:rho0},
one has $\rho_{\mathfrak{q}}(Y)>0$ as soon as $Y\ne0$.
Hence the real number
\begin{equation}\label{eqn:pX}
p_{G/H}:= \max_{Y\in\mathfrak{a}\setminus\{0\}}
   \frac{\rho_{\mathfrak{g}}(Y)}{\rho_{\mathfrak{q}}(Y)}
\end{equation}
is finite.

\begin{proposition}\label{prop:LpX} 
Let $G$ be an algebraic reductive Lie group
and $H$ an algebraic reductive subgroup of $G$
such that the action of $G$ on $G/H$ has compact kernel.
For any real $p\geq 1$, one has the equivalence~:\\ 
\centerline{$p>p_{G/H}$ $\Longleftrightarrow$
$\operatorname{vol}(gC\cap C) \in L^p(G)$
for any compact set $C$ in $G/H$.}
\end{proposition}

In this section we will show how to deduce Proposition \ref{prop:LpX}
from a volume estimate that we will prove in the following sections.
For that we will use another equivalent definition of the constant $p_{G/H}$.
\vspace{1em}

We extend $\mathfrak{a}$ to a maximal split abelian subspace
$\mathfrak{a}_{\mathfrak{g}}$ of $\mathfrak{g}$ and we choose 
a maximal compact subgroup  $K$ of $G$ such that
$H_K:=H\cap K$ is a maximal compact subgroup of $H$,
and that
$G=K(\exp\mathfrak{a}_{\mathfrak{g}})K$ and
$H=H_K(\exp\mathfrak{a}) H_K$
are Cartan decompositions of $G$ and $H$, respectively.

Let $W_G$ be the finite group 
$W_G:=N_G(\mathfrak{a}_{\mathfrak{g}})/Z_G(\mathfrak{a}_{\mathfrak{g}})
\simeq  N_K(\mathfrak{a}_{\mathfrak{g}})/Z_K(\mathfrak{a}_{\mathfrak{g}})$.
When $G$ is connected, $W_G$ is
the Weyl group of the restricted root system
$\Sigma(\mathfrak{g},\mathfrak{a}_{\mathfrak{g}})$.

For $Y \in \mathfrak{a}$,
we define a subset of $W_G$ by
\begin{equation}\label{eqn:WYa}
W(Y;\mathfrak{a}) := \{ w \in W_G: wY \in \mathfrak{a} \}.
\end{equation}
We notice that $W(Y;\mathfrak{a}) \ni e$ for any $Y \in \mathfrak{a}$,
and $W(0;\mathfrak{a}) = W_G$.
We set
\begin{align}
\rho_{\mathfrak{q}}^{\min}(Y)
&:= \min_{w\in W(Y;\mathfrak{a})} \rho_{\mathfrak{q}}(wY).
\label{eqn:rhomin}
\end{align}
We can then rewrite  Definition \eqref{eqn:pX} by the equivalent formula 
\begin{equation}\label{eqn:pX2}
p_{G/H}= \max_{Y\in\mathfrak{a}\setminus\{0\}}
   \frac{\rho_{\mathfrak{g}}(Y)}{\rho_{\mathfrak{q}}^{\min}(Y)}.
\end{equation}

\begin{proof}[Proof of Proposition \protect\ref{prop:LpX}]
The Haar measure $dg$ on $G$ is given as
\begin{eqnarray}\label{eqn:haarg}
\int_G f(g) d g  = \int_{ \mathfrak{a}_{\mathfrak{g}} } f ( e^Y ) 
D_{\mathfrak{g}}(Y)  d Y ,
\end{eqnarray}
for any $K$-biinvariant measurable function $f$ on $G$,
where $D_{\mathfrak{g}}$ is the $W_G$-invariant function on $\mathfrak{a}_{\mathfrak{g}}$ given by
\begin{eqnarray*}\label{eqn:wycc}
D_{\mathfrak{g}}(Y)
:= \prod_{\alpha\in\Sigma^{^+}\! (\mathfrak{g},\mathfrak{a}_{\mathfrak{g}})} 
   |\sinh \langle\alpha,Y\rangle|^{\dim\mathfrak{g}_\alpha},
\quad
Y \in \mathfrak{a}_{\mathfrak{g}}.
\end{eqnarray*}
and where
$\mathfrak{g}_\alpha\subset \mathfrak{g}$
are the (restricted) root spaces.

We also introduce the function on $\mathfrak{a}_{\mathfrak{g}}$
$$
\widetilde{D}_{\mathfrak{g}}(Y) :=\int_{\| Z\|\leq 1}D_{\mathfrak{g}}(Y+Z)dZ.
$$

We shall prove successively  the following equivalences
\begin{eqnarray*} 
&(i)& \text{ $\operatorname{vol}(gC\cap C)\in L^p(G)$, 
for any  compact $C\subset X$,}\\
\Longleftrightarrow&(ii)&
\text{$\operatorname{vol}(e^YC\cap C)^p\, D_{\mathfrak{g}}(Y)
\in L^1(\mathfrak{a}_{\mathfrak{g}})$, for any  compact $C\subset X$,}\\
\Longleftrightarrow&(iii)&
\text{$\operatorname{vol}(e^YC\cap C)^p\, \widetilde{D}_{\mathfrak{g}}(Y)
\in L^1(\mathfrak{a}_{\mathfrak{g}})$, for any  compact $C\subset X$,}\\
\Longleftrightarrow&(iv)&
\text{$\operatorname{vol}(e^YC\cap C)^p \, e^{\rho_{\mathfrak{g}}(Y)}
\in L^1(\mathfrak{a}_{\mathfrak{g}})$, for any  compact $C\subset X$,}\\
\Longleftrightarrow&(v)&
\text{$\operatorname{vol}(e^YC\cap C)^p\,  e^{\rho_{\mathfrak{g}}(Y)}
\in L^1(\mathfrak{a})$, for any  compact $C\subset X$,}\\
\Longleftrightarrow&(vi)&
\text{$e^{\rho_{\mathfrak{g}}(Y)-p\,\rho_{\mathfrak{q}}^{\min}(Y)}
\in L^1(\mathfrak{a})$,}\\
\Longleftrightarrow&(vii)&
\text{$p\,\rho_{\mathfrak{q}}^{\min}(Y)-\rho_{\mathfrak{g}}(Y)>0$, 
for any $Y\in \mathfrak{a}\smallsetminus 0$.}
\end{eqnarray*}
$(i)\Longleftrightarrow (ii)$ We may choose $C$ to be  $K$-invariant. 
We apply then
the integration formula  \eqref{eqn:haarg} to
the $K$-biinvariant function
$\operatorname{vol}(gC\cap C)$.\\
$(ii)\Longleftrightarrow (iii)$ We just replace $C$ by a larger compact $C':=e^{\mathfrak{a}_{\mathfrak{g}}(1)}C$
where $\mathfrak{a}_{\mathfrak{g}}(1)$ is the unit ball $\{ Z\in \mathfrak{a}_{\mathfrak{g}}: \| Z\|\leq 1\}$.\\
$(iii)\Longleftrightarrow (iv)$ We notice that 
we can find constants $a_1,a_2>0$, such that for any $Y\in \mathfrak{a}_{\mathfrak{g}}$,
one has 
\[
a_1 
\,e^{\rho_{\mathfrak{g}}(Y)}
\le \widetilde{D}_{\mathfrak{g}}(Y)\le  a_2 \,e^{\rho_{\mathfrak{g}}(Y)}.
\]
$(iv)\Longleftrightarrow (v)$ 
The main point of this equivalence is to replace an integration on  $\mathfrak{a}_\mathfrak{g}$ 
by an integration on $\mathfrak{a}$.
For that we will bound the support of the function $\varphi_C$ on
$\mathfrak{a}_{\mathfrak{g}}$, 
$\varphi_C(Y):= \operatorname{vol}(e^Y C \cap C)^p\, 
e^{\rho_{\mathfrak{g}}(Y)}$.
We may choose $C$ to be $K$-invariant so that, 
$\varphi_C$ is $W_G$-invariant.
We recall now the Cartan projection
\[
\mu\colon G \to \mathfrak{a}_{\mathfrak{g}} / W_G,
\ \ 
k_1 e^Y k_2 \mapsto Y \bmod W_G
\]
with respect to the Cartan decomposition
$G= K(\exp\mathfrak{a}_{\mathfrak{g}})K$.
It follows from  either \cite[Prop. 5.1]{Be} or \cite[Th. 1.1]{Ko} that, for any compact subsets
$S\subset G$, there exists $\delta>0$ such that
\begin{equation}\label{eqn:muSH}
\mu(S H S^{-1}) \subset
\mu(H) + \mathfrak{a}_{\mathfrak{g}}(\delta) \bmod W_G
\end{equation}
 where $\mathfrak{a}_{\mathfrak{g}}(\delta)$ stands for the $\delta$-ball 
$\{Y \in \mathfrak{a}_{\mathfrak{g}}: \|Y\| \le \delta\}$.  
If we take this compact set $S \subset G$ such that
$C \subset S H/H$,
then $Y\in\mathfrak{a}_{\mathfrak{g}}$ satisfies
$e^Y C\cap C\ne\emptyset$
only if
$e^Y\in S H S^{-1}$,
and therefore,
only if $Y\in\mu(SHS^{-1})$.
Hence  we get the bound on the support
\begin{equation}\label{eqn:proper}
   {\operatorname{Supp}} \, \varphi_C
  \subset 
  \bigcup_{w \in W_G} w ({\mathfrak {a}} + \mathfrak{a}_{\mathfrak{g}}(\delta)).
\end{equation} 
By $W_G$-invariance of $\varphi_C$, 
we only have to integrate on the $\delta$-neighborhood of $\mathfrak{a}$.
Hence the assertion $(iv)$ is equivalent to the following assertion 
$$
(iv')\ 
\operatorname{vol}(e^YC\cap C)^p \, 
e^{\rho_{\mathfrak{g}}(Y)}\in 
L^1( \mathfrak {a}+\mathfrak{a}_{\mathfrak{g}}(R))
\text{ for any compact $C\subset X$, $R>0$.}
$$
To see that this assertion $(iv')$ is equivalent to $(v)$, we just 
have, for both implications,  to replace the compact $C$ by a larger compact 
$C':=e^{\mathfrak{a}_{\mathfrak{g}}(R)}C$ and to notice that the map 
$\displaystyle Y\mapsto \max_{Z\in \mathfrak{a}_{\mathfrak{g}}(R)}
|\rho_{\mathfrak{g}}(Y+Z)-\rho_{\mathfrak{g}}(Y)|$ 
is uniformly bounded on $\mathfrak{a}$.
\\
$(v)\Longleftrightarrow (vi)$ 
We use Theorem \ref{thm:volX}, that we will prove in the next section, 
and that gives, for $C$ large enough,   constants 
$m, M>0$ such that
\[
m \, e^{-\rho_{\mathfrak{q}}^{\min}(Y)}
\le \operatorname{vol}(e^YC\cap C) 
\le M \, e^{-\rho_{\mathfrak{q}}^{\min}(Y)}
\quad\text{for any 
$Y\in \mathfrak{a}$}.
\]
$(vi)\Longleftrightarrow (vii)$ We recall that the function $\rho_{\mathfrak{g}}\! -\! p\rho_{\mathfrak{q}}^{\min}$ is 
continuous and piecewise linear.

This proves Proposition \ref{prop:LpX} provided the following Theorem \ref{thm:volX}.
\end{proof}

\subsection{Estimate of $\operatorname{vol}(e^Y C \cap C)$}
\label{subsec:volC}

Let $C$  be a compact subset of $X$.
We shall give both lower and upper bounds of the volume of
$e^Y C\cap C$
as $Y\in\mathfrak{a}$ goes to infinity.
For that we will use the function $\rho_\mathfrak{q}^{\min} $
defined by formula  \eqref{eqn:rhomin}.
Let $x_0 = e H \in X = G/H$ and $W_Gx_0$ be the orbit of this point 
under the Weyl group of $G$.

\begin{theorem}\label{thm:volX}
Let $G$ be an algebraic reductive Lie group,
$H$ an algebraic reductive subgroup and
$C$  a compact neighborhood of $Kx_0$ in $X:=G/H$.
Then there exist constants $m\equiv m_C>0$ and $M\equiv M_C>0$ such that
\[
me^{-\rho_{\mathfrak{q}}^{\min}(Y)}
\le \operatorname{vol}(e^Y C \cap C)
\le Me^{-\rho_{\mathfrak{q}}^{\min}(Y)}
\ \ \text{for any $Y\in\mathfrak{a}$}.
\]
\end{theorem}

The proof of the lower bound will be given in Section \ref{sec:CC}.

We will give the proof of the upper bound in eight steps which 
will last from Section \ref{sec:CC} to \ref{sec:ubv}.
Clearly, the upper bound in Theorem \ref{thm:volX}  is equivalent to the following
statement: \\
{\it For any compact sets $C_1$, $C_2$ in $X$,
there exists
$M \equiv M_{C_1,C_2} > 0$
such that
\begin{equation}\label{eqn:vol2}
\operatorname{vol}(e^Y C_1 \cap C_2)
\le Me^{-\rho_{\mathfrak{q}}^{\min}(Y)}
\quad\text{for any $Y\in\mathfrak{a}$}.
\end{equation}
}

The strategy of the proof of \eqref{eqn:vol2} will be to 
see $G/H$ as  a closed orbit 
in a representation of $G$ and to
decompose $C_1$ and $C_2$ into smaller compact pieces.

\subsection{Lower bound for  $\operatorname{vol}(e^Y C \cap C)$}
\label{sec:CC}

Up to the end of this chapter we keep the setting as above~: 
$G$ is a connected algebraic reductive Lie group,
and $H$ an algebraic reductive subgroup. 

By Chevalley theorem (\cite[Th. 5.1]{Bo69}
or \cite[Section 4.2]{Che}), there exists an algebraic representation
$\tau\colon G\to GL(V)$ such that the homogeneous space $X=G/H$ is realized
 as a closed orbit  $X=Gx_0\subset V$ where ${\rm Stab}_G(x_0)=H$.
We can assume that $\operatorname{Ker}(d\tau)=\{0\}$.
We fix such a representation $(\tau,V)$ once and for all.
\vspace{1em}

Here is the first step towards both the  volume upper bound 
and the volume lower bound in Theorem \ref{thm:volX}.

\begin{lemma}
\label{lem:CC}
There exists a neighborhood\/ $C_{x_0}$ 
of $x_0$ in $G/H$
such that for any compact neighborhood $C_0$ of $x_0$ contained
in $C_{x_0}$,
there exist constants $m,M>0$ such that 
\begin{equation}\label{eqn:CC}
m e^{-\rho_{\mathfrak{q}}(Y)} 
\le \operatorname{vol}(e^Y C_0 \cap C_0) \le M e^{-\rho_{\mathfrak{q}}(Y)}
\quad\text{for any $Y\in\mathfrak{a}$}.  
\end{equation}
\end{lemma}

\begin{proof}
Since $G$ and $H$ are reductive, the representation of $H$ in $\mathfrak{q}$ is 
volume preserving. Hence
we can apply Proposition \ref{prop:volV} to the representation of $H$ in $\mathfrak{q}$.
Roughly, the strategy  is then to linearize $X$ near $x_0$. 
To make this approach precise, 
we need two similar but slightly different arguments for the lower bound and for the upper bound.
\vspace{1em}

{\it Lower bound}. 
We choose a sufficiently small compact neighborhood $U_0$ of $0$ in $\mathfrak{q}$ 
on which the map 
$$
\pi_-\colon \mathfrak{q}\rightarrow X\; ,\; Z\mapsto  e^Z x_0
$$ 
is well-defined, injective with a Jacobian bounded away from $0$.
Since $x_0$ is $H$-invariant, this map $\pi_-$ is $H$-equivariant.
For any compact neighborhood $C_0=\pi_-(C)$ of $x_0$ in $X$ with $C\subset U_0$,
one has, for every $Y\in \mathfrak{a}$,
\[
e^YC_0\cap C_0 \supset \pi_-(e^Y C\cap  C).
\]
The lower bound  in \eqref{eqn:CC} is then a consequence of 
the lower bound in Proposition \ref{prop:volV}.

\vspace{1em}

{\it Upper bound}. 
Since the linear tangent space $T_{x_0}X\subset V$ of $X$ at $x_0$
is canonically $H$-isomorphic to $\mathfrak{q}$,
we will also denote it by $\mathfrak{q}$.
Since $H$ is reductive, this vector subspace $\mathfrak{q}\subset V$
admits a  $H$-invariant supplementary subspace $\mathfrak{s}$.
We set $p\colon V\rightarrow \mathfrak{q}$ for the  linear projector with kernel $\mathfrak{s}$.
We choose a sufficiently small compact neighborhood $C_{x_0}$ of $x_0$ in $X$ 
on which the map 
$$
\pi_+ \colon X\rightarrow \mathfrak{q} \; ,\; x\mapsto p(x)-p(x_0)
$$
is injective with a Jacobian bounded away from $0$.
Since $x_0$ is $H$-invariant, this map $\pi_+$ is also $H$-equivariant.

For any compact subset $C_0$ of $C_{x_0}$, one has, for every $Y\in \mathfrak{a}$,
\[
\pi_+(e^YC_0\cap C_0) \subset e^Y C\cap  C
\]
where $C:=\pi_+(C_0)$.
The upper bound in  \eqref{eqn:CC} is then a consequence of 
the upper bound  in Proposition \ref{prop:volV}.
\end{proof}

As a direct corollary we get the lower bound in Theorem \ref{thm:volX}~.

\begin{cor}
\label{cor:CC}
For any compact neighborhood\/ $C$ of\/ $Kx_0$ in $G/H$,
there exists $m>0$  such that 
$$
\operatorname{vol}(e^Y C \cap C) \ge m e^{-\rho^{\min}_{\mathfrak{q}}(Y)}
 \quad\text{for any $Y\in\mathfrak{a}$}.  
$$
\end{cor}

\begin{proof} Shrinking $C$ if necessary,
we can assume that $C=K C_{0}$ where $C_{0}$ is a compact neighborhood of 
$x_0$. 
According to Lemma \ref{lem:CC}, there exists a constant $m>0$ 
such that the lower bound in \eqref{eqn:CC} is satisfied.
For each
$w \in W(Y; \mathfrak{a})$ ($\subset W_G$),
we take a representative
$k_w \in N_K (\mathfrak{a}_{\mathfrak{g}})$.
Then one has
\begin{equation*}
\operatorname{vol} (e^Y C \cap C)
\ge \operatorname{vol}(e^Y k_w^{-1} C_0 \cap C_0)
  = \operatorname{vol} (e^{wY} C_0 \cap C_0)
 \ge m e^{-\rho_{\mathfrak{q}}(wY)}.
\end{equation*}
Hence one has
\begin{eqnarray*}
\operatorname{vol}(e^Y C \cap C)
&\geq &
m \,\max_{w\in W(Y;\mathfrak{a})} e^{-\rho_{\mathfrak{q}}(wY)}
\;\; =\;\; m  \, e^{-\rho^{\min}_{\mathfrak{q}}(Y)}.
\end{eqnarray*}
This ends the proof.
\end{proof}

\subsection{Volume near one invariant point}
\label{sec:V0}

Here is the second step towards the  volume upper bound \eqref{eqn:vol2}.
It is a subtle variation of the volume upper bound given in  Lemma \ref{lem:CC}.

For any   subspace  $\mathfrak{b}\subset \mathfrak{a}$, we set
$X^{\mathfrak{b}}:=\{x\in X: e^{Y}x=x,\;\text{for all $Y\in\mathfrak{b}$}\}$.

\begin{lemma}
\label{lem:volC2} For any subspace  
$\mathfrak{b}\subset \mathfrak{a}$ and any
point $x \in X ^{\mathfrak{b}}$,
there exists a neighborhood $C_x$ of $x$ in $X$ and $M>0$ such that 
\[
\operatorname{vol}(e^Y C_x \cap C_x)
\le M \, e^{-\rho_{\mathfrak{q}}^{\min} (Y)}
 \quad\text{for any $Y\in\mathfrak{b}$}.
\]

\end{lemma}

\begin{proof}[Proof of Lemma \ref{lem:volC2}]
Let $H'$ be the stabilizer of $x$ in $G$ and $\mathfrak{h}'$ its Lie algebra.
Since $x$ is in $ X^{\mathfrak{b}}$,
one has ${\mathfrak{b}}\subset \mathfrak{h}'$. 
Hence there exists  a maximal split abelian subspace $\mathfrak{a}' $ of $\mathfrak{h}'$ 
containing ${\mathfrak{b}}$.
Since all the maximal split abelian subspaces of $\mathfrak{h}$ are $H$-conjugate,
one can find an element $g \in G$ such that 
$x = g x_0$. 
Then one has
$
H' := gHg^{-1}
$ and
$\mathfrak{h}' := \operatorname{Ad}(g)\mathfrak{h}$.
After replacing $g$ by a suitable element $gh$ with $h$ in $H$, we  also have
$\mathfrak{a}' = \operatorname{Ad}(g)\mathfrak{a}.
$
We set
$\mathfrak{q}' := \operatorname{Ad}(g)\mathfrak{q}$
and introduce the function
$\rho_{\mathfrak{q}'} \colon
 \mathfrak{a}' \to \mathbb{R}_{\ge0}$
associated to
the representation of $H'$ on
$ \mathfrak{q}'$ as in Section \ref{subsec:rhoV}.
By definition,
we have the following identity:
\begin{equation}\label{eqn:rhoqq}
\rho_{\mathfrak{q}'}(\operatorname{Ad}(g)Z) = \rho_{\mathfrak{q}}(Z)
\quad\text{for any $Z \in \mathfrak{a}$}.
\end{equation}

Applying Lemma \ref{lem:CC}  to the homogeneous space $G/H'$,
we see that there exist
a compact neighborhood $C_x$ of $x$ in $X$ and a constant
$ M > 0$ such that
\begin{equation}\label{eqn:volgprime}
 \operatorname{vol}(e^Y C_x \cap C_x)
\le M e^{-\rho_{\mathfrak{q}'}(Y)}
\ \text{for any $Y\in\mathfrak{a}'$}.
\end{equation}

Now, for  $Y \in {\mathfrak{b}}$, we set $Z = \operatorname{Ad}(g^{-1})Y$.
This element $Z$ belongs also to $\mathfrak{a}$.
Since the Cartan subspace $\mathfrak{a}_{\mathfrak{g}}$
contains $\mathfrak{a}$ and since two elements of $\mathfrak{a}_{\mathfrak{g}}$ 
which are $G$-conjugate are  always $W_G$-conjugate, 
there exists $w\in W_G$ such that $Z=wY$.
Using 
\eqref{eqn:rhoqq},
we get
$$
\rho_{\mathfrak{q}'} (Y)
 = \rho_{\mathfrak{q}} (Z)
=\rho_{\mathfrak{q}} (wY)
\geq \rho_{\mathfrak{q}}^{\min}(Y).
$$
Hence, Lemma \ref{lem:volC2} follows from \eqref{eqn:volgprime}.
\end{proof}

\subsection{Volume near two invariant points}

Here is the third step towards the  volume upper bound \eqref{eqn:vol2}.

\begin{lemma}\label{lem:volC3}
For any  vector subspace  $\mathfrak{b}\subset\mathfrak{a}$  and any points
$x_1$, $x_2$ in $X^\mathfrak{b}$, 
there exist compact neighborhoods  $C_1$ of $x_1$
and $C_2$ of $x_2$ in $X$, and $ M>0$
such that
\begin{eqnarray}\label{eqn:volC3}
\operatorname{vol}(e^Y C_1 \cap C_2) \le M\, e^{-\rho_{\mathfrak{q}}^{\min}(Y)}
\quad\text{for any $Y \in \mathfrak{b}$}.
\end{eqnarray}
\end{lemma}

We set 
\begin{eqnarray}\label{eqn:vbvvbv}
V^\mathfrak{b}:=\{ v\in V: \mathfrak{b}v=0\}
\end{eqnarray}
so that 
$X^\mathfrak{b}=X\cap V^\mathfrak{b}$ 
and we set $\pi^\mathfrak{b}\colon V\rightarrow V^\mathfrak{b}$ to be the 
$\mathfrak{b}$-equivariant projection.

\begin{proof}
When $x_1=x_2$ this is Lemma \ref{lem:volC2}. 
When $x_1\neq x_2$,
we choose $C_1$ and $C_2$ with 
disjoint projections $\pi^\mathfrak{b}(C_1)
\cap\pi^\mathfrak{b}(C_2)=\emptyset$ so that, 
for any $Y$ in $\mathfrak{b}$, the intersection
$e^Y C_1 \cap C_2$ is also empty.
\end{proof}

Here is the fourth step towards the  volume upper bound \eqref{eqn:vol2}.

\begin{lemma}\label{lem:volC4}
For any  vector subspace  $\mathfrak{b}\subset\mathfrak{a}$  and any compact subsets
$S_1$, $S_2$ included in $X^\mathfrak{b}$, there exist $M>0$ and 
compact neighborhoods  
$C_{S_1}$ of $S_1$
and $C_{S_2}$ of $S_2$ in $X$
such that
\begin{eqnarray}\label{eqn:volC4}
\operatorname{vol}(e^Y C_{S_1} \cap C_{S_2}) \le 
M\, e^{-\rho_{\mathfrak{q}}^{\min}(Y)}
\quad\text{for any $Y \in \mathfrak{b}$}.
\end{eqnarray}
\end{lemma}

\begin{proof} This is a consequence of Lemma \ref{lem:volC3}
by a standard compactness argument.
Let $x_1\in S_1$. For any  $x_2\in S_2$, there exist compact neighborhoods 
$C_1(x_1,x_2)$ of $x_1$  and 
$C_2(x_1,x_2)$ of $x_2$ satisfying \eqref{eqn:volC3}.

First we fix $x_1$ in $S_1$. By compactness of $C_2$, one can find a finite set 
$F_2\equiv F_2(x_1)\subset S_2$ for which 
the union $C_2(x_1,S_2):=\cup_{x_2\in F_2}C_2(x_1,x_2)$
is a compact neighborhood of $S_2$. The intersection 
$C_1(x_1,S_2):=\cap_{x_2\in F_2}C_1(x_1,x_2)$ is still 
a compact neighborhood of $x_1$.

By compactness of $C_1$, one can find a finite set 
$F_1\subset S_1$ for which 
the union $C_{S_1}:=\cup_{x_1\in F_1}C_1(x_1,S_2)$
is a compact neighborhood of $S_1$. The intersection 
$C_{S_2}:=\cap_{x_1\in F_1}C_2(x_1,S_2)$ is still 
a compact neighborhood of $S_2$.

Since only finitely many constants $M$ are involved in this process,
the compact neighborhoods  
$C_{S_1}$
and $C_{S_2}$ 
satisfy \eqref{eqn:volC4}
\end{proof}

\subsection{Facets}
\label{sec:fac}

In this section,
we shall introduce a decomposition 
of $\mathfrak{a}$ in convex pieces
$F$ called facets by
using the representation
$d\tau|_{\mathfrak{a}}\colon \mathfrak{a}\to\operatorname{End}(V)$.

We need to introduce more notations.
Let $\Delta\equiv\Delta(V,\mathfrak{a})$ be the set of weights
of $\mathfrak{a}$ in $V$.
For $v$ in $V$ we write  
$\displaystyle v = \sum_{\lambda\in\Delta} v_\lambda$
according to the weight space decomposition 
$\displaystyle V=\bigoplus_{\lambda\in\Delta}V_\lambda$.
We fix a norm $\| \ \|$ on each weight space
$V_\lambda$,
and define a norm on $V$ by
\begin{equation}\label{eqn:Ve}
\|v\| := \max_{\lambda\in\Delta} \|v_\lambda\|.
\end{equation}

For any  subset  $F\subset\mathfrak{a}$, we set
\begin{eqnarray*}\label{eqn:def}
\Delta^+_F &:=&
\{\lambda\in \Delta: 
\lambda(Y)>0\ \text{for any $Y \in F$} \} \\
\Delta^0_F &:=&
\{\lambda\in \Delta: 
\lambda(Y)=0\ \text{for any $Y \in F$} \} \\
\Delta^-_F &:=&
\{\lambda\in \Delta: 
\lambda(Y)<0\ \text{for any $Y \in F$} \} 
\end{eqnarray*}
We say that $F$ is a {\it{facet}} if 
\begin{equation*}
\label{eqn:Deltapm}
  \Delta = \Delta_F^+ \amalg \Delta_F^0 \amalg \Delta_F^- 
  \;\;\;\;\;\;{\rm and}
\end{equation*}
\begin{eqnarray*}\label{eqn:FYa}
F=\{Y \in {\mathfrak {a}}:
&&
\lambda(Y) >0\ \text{for any $\lambda \in \Delta_F^+$}, \\
\nonumber &&
\lambda(Y) =0\ \text{for any $\lambda \in \Delta_F^0$} ,\\
\nonumber &&
\lambda(Y) <0\ \text{for any $\lambda \in \Delta_F^-$} \}.
\end{eqnarray*}

Let ${\mathcal{F}}$ be the totality of facets.  
Then we have
$$
{\mathfrak {a}} = \bigsqcup_{F \in {\mathcal {F}}} F
\quad
\text{(disjoint union)}.  
$$

For any facet $F$ we denote by $\mathfrak{a}_F$ its support, 
i.e. its linear span:
\begin{eqnarray*}\label{eqn:aF}
\mathfrak{a}_F
:= \{ Y \in \mathfrak{a}:  \lambda(Y) = 0
      \ \ \text{for any $\lambda \in \Delta_F^0$} \}.
\end{eqnarray*}

We set
\[
V_F^\varepsilon := \bigoplus_{\lambda\in\Delta_F^\varepsilon} V_\lambda
\quad\text{for $\varepsilon = +,0,-$}.
\]
Notice that, using Notation \eqref{eqn:vbvvbv}, we have $V_F^0=V^{{\mathfrak a}_F}$.
We have a direct sum decomposition:
\begin{equation}\label{eqn:VF}
V = V_F^+ \oplus V_F^0 \oplus V_F^-.
\end{equation}

Here is the fifth step towards the  volume upper bound \eqref{eqn:vol2}.

\begin{lemma}\label{lem:volC5}
Let $F$ be a facet, $S_1$ be a compact subset
of $X\cap (V_F^0\oplus V_F^-)$,
and  $S_2$ be a compact subset
of $X\cap (V_F^0\oplus V_F^+)$.
Then there exist 
compact neighborhoods  
$C_{S_1}$ of $S_1$
and $C_{S_2}$ of $S_2$ in $X$, and $M>0$
such that
\begin{eqnarray}\label{eqn:volC5}
\operatorname{vol}(e^Y C_{S_1} \cap C_{S_2}) \le 
M\, e^{-\rho_{\mathfrak{q}}^{\min}(Y)}
\quad\text{for any $Y \in \mathfrak{a}_F$}.
\end{eqnarray}
\end{lemma}

\begin{proof}
We recall that $\pi^{\mathfrak{a}_F}$ is the projection on 
$V_F^0=V^{{\mathfrak a}_F}$. 
Since $X$ is closed and is invariant under the group $e^{\mathfrak{a}_F}$,
one has the inclusions
\begin{eqnarray*}\label{eqn:piaf}
\pi^{\mathfrak{a}_F}( X\cap (V_F^0\oplus V_F^-))\subset X^{\mathfrak{a}_F}
\;\;\;{\rm and}\;\;\;
\pi^{\mathfrak{a}_F}( X\cap (V_F^0\oplus V_F^+))\subset X^{\mathfrak{a}_F}.
\end{eqnarray*}

Let $T_1:=\pi^{\mathfrak{a}_F}(S_1)$ and 
$T_2:=\pi^{\mathfrak{a}_F}(S_2)$ be the images of $S_1$ and $S_2$
by the projection $\pi^{\mathfrak{a}_F}$. Since 
\begin{eqnarray}\label{eqn:sxvf}
S_1\subset X\cap (V_F^0\oplus V_F^-)
\;\;\;{\rm and}\;\;\;
S_2\subset X\cap (V_F^0\oplus V_F^+),
\end{eqnarray}
these images $T_1$ and $T_2$ are compact subsets of $X^{\mathfrak{a}_F}$.
According to Lemma \ref{lem:volC4} with $\mathfrak{b}=\mathfrak{a}_F$, 
there exist $M>0$ and 
compact neighborhoods  
$C_{T_1}$ of $T_1$
and $C_{T_2}$ of $T_2$ in $X$
such that 
\begin{eqnarray}\label{eqn:419}
\operatorname{vol}(e^Y C_{T_1} \cap C_{T_2}) \le 
M\, e^{-\rho_{\mathfrak{q}}^{\min}(Y)}
\quad\text{for any $Y \in \mathfrak{a}_F$}.
\end{eqnarray}

Using again \eqref{eqn:sxvf}, one can then find an element $Y_0\in F$ such that 
\begin{eqnarray*}\label{eqn:eysc}
e^{Y_0}S_1\subset \text{interior of }C_{T_1}
\;\;\;{\rm and}\;\;\;
e^{-Y_0}S_2\subset \text{interior of }C_{T_2}.
\end{eqnarray*}
We choose then the neighborhoods
\begin{eqnarray*}\label{eqn:csct}
C_{S_1}:= e^{-Y_0}C_{T_1}
\;\;\;{\rm and}\;\;\;
C_{S_2}:= e^{Y_0}C_{T_2}.
\end{eqnarray*}
of $S_1$ and $S_2$ respectively. 
According to \eqref{eqn:419},  one has, for any $Y\in \mathfrak{a}_F$,
\begin{eqnarray*}\label{eqn:volC4bis}
\operatorname{vol}(e^Y C_{S_1} \cap C_{S_2}) 
= \operatorname{vol}(e^{Y-2Y_0} C_{T_1} \cap C_{T_2})
\le 
M\, e^{-\rho_{\mathfrak{q}}^{\min}(Y-2Y_0)}.
\end{eqnarray*}
Since the function $Y\mapsto |\rho_{\mathfrak{q}}^{\min}(Y\! -\! 2Y_0)\! -\! \rho_{\mathfrak{q}}^{\min}(Y)|$
is uniformly bounded on $\mathfrak{a}$, this gives the volume upper bound
\eqref{eqn:volC5}.
\end{proof}

Here is the sixth step towards the  volume upper bound \eqref{eqn:vol2}.

\begin{lemma}\label{lem:volC6}
Let $F$ be a facet and 
$C_1$, $C_2$ compact subsets of $G/H$.
Suppose   $C_1\cap(V_F^0\oplus V_F^-)=\emptyset$
or $C_2\cap(V_F^0\oplus V_F^+)=\emptyset$.
Then there exists $Y_0\in F$ such that 
\begin{equation*}\label{eqn:volC6}
e^Y C_1 \cap C_2 = \emptyset
\ \ 
\text{for any\/ $Y \in Y_0+ F$}.
\end{equation*}
\end{lemma}

\begin{proof}
For a compact subset $C$ of $X$ and  $\lambda \in \Delta$,
 we set 
\begin{eqnarray*}\label{eqn:mlac}
  m_{\lambda}(C) :=\min_{v \in C} \|v_{\lambda}\|
\;\;\;{\rm and}\;\;\;
  M_{\lambda}(C) :=\max_{v \in C} \|v_{\lambda}\|,
\end{eqnarray*}
and for
$\varepsilon=\pm$, 
we set
\begin{eqnarray*}\label{eqn:mfec}
m_F^\varepsilon(C) := 
  \max_{\lambda\in\Delta_F^\varepsilon}m_\lambda(C)
\;\;\;{\rm and }\;\;\;
M_F^\varepsilon(C) := 
  \max_{\lambda\in\Delta_F^\varepsilon}M_\lambda(C).
\end{eqnarray*}

If $C_1\cap(V_F^0\oplus V_F^-)=\emptyset$, one has $m_F^+(C_1)>0$ 
and we choose $Y_0\in F$ such that, for all $\lambda\in \Delta_F^+$,
$$
e^{\lambda(Y_0)}> \frac{M_F^+(C_2)}{m_F^+(C_1)}.
$$
Let $Y \in Y_0+F$. 
By definition of $m_F^+(C_1)$, 
one can find $\lambda\in \Delta^+_F$ such that, 
for any $v$ in $C_1$, one has  $\|v_\lambda\|\geq m_F^+(C_1)$.
One has then 
$$
\| (e^Yv)_\lambda\|=e^{\lambda(Y)}\| v_\lambda\|
\geq e^{\lambda(Y_0)}m_F^+(C_1)>M_F^+(C_2).
$$
Hence $e^Yv$ does not belong to $C_2$. This proves 
that $e^YC_1\cap C_2=\emptyset$.

Likewise, if $C_2\cap(V_F^+\oplus V_F^0)=\emptyset$, one has $m_F^-(C_2)>0$,
and we choose $Y_0\in F$ such that, for all $\lambda\in \Delta_F^-$,

$\displaystyle \mbox{ } \hspace{11em}
e^{-\lambda(Y_0)}> \frac{M_F^-(C_1)}{m_F^-(C_2)}.
$
\end{proof}

\subsection{Upper bound for  $\operatorname{vol}(e^Y C \cap C)$}
\label{sec:ubv}

Here is the seventh step towards the  volume upper bound \eqref{eqn:vol2}.
For any facet $F$, any $Y_0\in F$, and any $R\geq 0$, we introduce
the $R$-neighborhood of the $Y_0$-translate of the facet $F$~:
\begin{eqnarray}\label{eqn:fyr}
F(Y_0,R):= Y_0+F+\mathfrak{a}(R),
\end{eqnarray}
where $\mathfrak{a}(R)$ is the ball 
$\{ Y\in \mathfrak{a}: \| Y\|\leq R\}$.

\begin{lemma}\label{lem:volC7}
Let $F$ be a facet, $R\geq 0$, and 
$C_1$, $C_2$ compact subsets of $G/H$.
Then there exist $Y_{F,R} \in F$ and $M>0$ such that 
\begin{eqnarray}\label{eqn:volC7}
\operatorname{vol}(e^Y C_1 \cap C_2) \le 
M\, e^{-\rho_{\mathfrak{q}}^{\min}(Y)}
\quad\text{for any $Y \in F(Y_{F,R},R)$}.
\end{eqnarray}
\end{lemma}

\begin{proof} We first assume that $R=0$.
We will deduce Lemma \ref{lem:volC7} from the two
previous steps, namely Lemmas  \ref{lem:volC5} and  \ref{lem:volC6}.
Let
\begin{eqnarray*}\label{eqn:scvf}
S_1:=C_1\cap (V_F^0\oplus V_F^-)
\;\;\;{\rm and}\;\;\;
S_2:=C_2\cap (V_F^0\oplus V_F^+).
\end{eqnarray*}
According to Lemma \ref{lem:volC5} we can write
\begin{eqnarray*}\label{eqn:ccsc}
C_1:=C_{S_1}\cup C'_1
\;\;\;{\rm and}\;\;\;
C_2:=C_{S_2}\cup C'_2
\end{eqnarray*}
where $C_{S_1}$ and $C_{S_2}$ are respectively compact neighborhoods of 
$S_1$ in $C_1$ and of 
$S_2$ in $C_2$ satisfying the volume upper bound \eqref{eqn:volC5}
for some constant $M>0$,
and where $C'_1$ and $C'_2$ are compact subsets of $X$ such that
\begin{eqnarray*}\label{eqn:cvfvf}
C'_1\cap (V_F^0\oplus V_F^-)=\emptyset
\;\;\;{\rm and}\;\;\;
C'_2\cap (V_F^0\oplus V_F^+)=\emptyset.
\end{eqnarray*}
Hence according to Lemma \ref{lem:volC6}, there exists an element 
$Y_F\in F$ such that, for any $Y\in Y_F+ F$, one has, 
\begin{equation*}\label{eqn:eycc}
e^Y C'_1 \cap C'_2 = e^Y C_{S_1} \cap C'_2 =e^Y C'_1 \cap C_{S_2} =\emptyset\;.
\end{equation*}
Hence, one has the desired volume upper bound,
for any $Y \in Y_F +F$,
\begin{eqnarray*}\label{eqn:volC70}
\operatorname{vol}(e^Y C_1 \cap C_2) = 
\operatorname{vol}(e^Y C_{S_1} \cap C_{S_2}) \le 
M\, e^{-\rho_{\mathfrak{q}}^{\min}(Y)}.
\end{eqnarray*}

When $R$ is not zero, we  apply 
the first case to the compact sets $e^{\mathfrak{a}(R)}C_1$ and $C_2$
and notice that the function 
$Y\mapsto \displaystyle\max_{Z\in\mathfrak{a}(R)} 
|\rho_{\mathfrak{q}}^{\min}(Y\! +\! Z)\! -\! \rho_{\mathfrak{q}}^{\min}(Y)|$
is uniformly bounded on $\mathfrak{a}$.
\end{proof}

\begin{proof}[Proof of Theorem \ref{thm:volX}]
Here is the eighth and last step towards the  volume upper bound \eqref{eqn:vol2}.
Fix two  compact sets
$C_1, C_2$ of $G/H$. According to Lemma \ref{lem:volC7},
{\it  given any facet $F\in\mathcal{F}$  and any $R>0$ there
exist $Y_{F,R}\in F$ and $M>0$ such that \eqref{eqn:vol2} holds for any
$Y\in F(Y_{F,R},R)$.} 
The following Lemma \ref{lem:afyr} tells us that \eqref{eqn:vol2} holds
for any $Y$ in $\mathfrak{a}$.
This ends the proof of the volume upper bound \eqref{eqn:vol2} 
and of  Theorem \ref{thm:volX},
\end{proof}

\begin{lemma}
\label{lem:afyr}
Assume that, for any facet $F$
and any 
$R\ge 0$, we are given an element $Y_{F,R}\in F$.
Then, one can choose for every facet $F$ a constant $R_F\ge 0$ 
such that, using notations \eqref{eqn:fyr},  one has 
\begin{eqnarray}\label{eqn:afl}
\mathfrak{a} =
\bigcup_{F\in\mathcal{F}}
F(Y_{F,R_F},R_F).
\end{eqnarray}
\end{lemma}

\begin{figure}
\centering
\begin{tikzpicture}[scale=0.3]
\foreach \i in {0,120,240}
{
  \begin{scope}[rotate=\i]
	\foreach \j in {-60,0,30,60}{
	  \draw[blue, rotate=\j] (0,0) -- (8,0);
  \draw[rotate=\j] (8,1)--(5,1) arc (90:270:1cm)--(8,-1);
  }
  \draw[shift={(-30:1cm)}] (0:7cm) --(0:0cm) -- (-60:7cm);
  \foreach \j in {0,30}{
	  \draw[rotate=\j,shift={(15:2cm)}] (0:6cm) --(0:0cm) -- (30:6cm);
  }
  \end{scope}
}
\draw (0,0) circle (6cm);
\end{tikzpicture}

\caption{Cover of $\mathfrak{a}$ }
\label{figcover}
\end{figure}
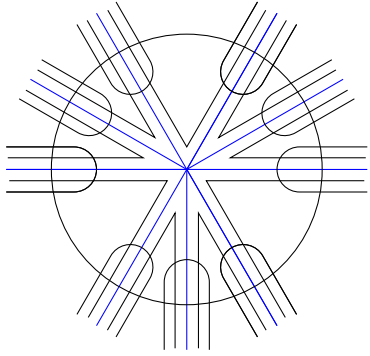

\begin{proof} We will choose inductively on 
$\ell=0, 1, \ldots , \dim\mathfrak{a}$,
simultaneously the constants $R_F$ 
for all the facets of codimension $\ell$ (see Figure \ref{figcover}). 

We first choose $R_F=0$ for all the open facets $F$.

We assume that $R_F$ has been chosen for the facets of codimension 
strictly less than $\ell$
and we consider the set 
$$\mathfrak{a}_\ell =
\bigcup_{\substack{F\in\mathcal{F}\\ \operatorname{codim} F<\ell}}
F(Y_{F,R_F},R_F).$$
We assume, by induction, that there exists a constant $\delta_\ell>0$ such that
the complementary set 
$
\mathfrak{a}\setminus \mathfrak{a}_\ell$ is  included in a $\delta_\ell$-neighborhood
of the union  of the facets of codimension $\ell$.
We choose $R_F= \delta_\ell$ for all the facets of codimension $\ell$.
This gives a new set $\mathfrak{a}_{\ell +1}$.
The complementary set 
$\mathfrak{a}\setminus \mathfrak{a}_{\ell +1}$ is  
then included in a $\delta_{\ell+1}$-neighborhood
of the union  of the facets of codimension $\ell+1$, for some constant $\delta_{\ell+1}>0$.
And we go on by induction.
\end{proof}

\section{Application}

The criterion given in Theorem \ref{thm:Xtemp} is easy to apply:
it is easy to detect for a given homogeneous space $G/H$ whether 
the unitary representation of $G$ in $L^2(G/H)$ is tempered or not. 
We collect in this chapter a few corollaries of this criterion,
omitting the details of the computational verifications that 
will be published elsewhere together with a complete classification
of homogeneous spaces $G/H$ for which $L^2(G/H)$ is non-tempered.

\subsection{Abelian or amenable generic stabilizer}
\label{seccor}

For general real reductive homogeneous spaces,
we deduce the following facts~:

\begin{pro} 
\label{progh1h2}
Let $p\geq 2$ be an even integer.
Let $G$ be a semisimple algebraic Lie group,
and
$H_1\supset H_2$ two unimodular subgroups.\\
$a)$ If $L^2(G/H_1)$ is almost $L^p$ then $L^2(G/H_2)$ is almost $L^p$.\\
$b)$ The converse is true when $H_2$ is normal in $H_1$ and $H_1/H_2$ is amenable
(for instance finite, or compact, or abelian).
\end{pro}

\begin{pro} 
\label{porgh}
Let $p\geq 2$ be an even integer.
Let $G$ be an algebraic semisimple Lie group,
and
$H$ an algebraic reductive subgroup.\\
$a)$ If the representation of $G_{\mathbb{C}}$ in
$L^2(G_{\mathbb{C}}/H_{\mathbb{C}})$ is almost $L^p$,
then the representation of $G$ in $L^2(G/H)$ is almost $L^p$.\\
$b)$ The converse is true when $H$ is a split group.
\end{pro}

\begin{theorem}\label{thXtemp2}
Let $G$ be an algebraic  semisimple real  Lie group,
and $H$ an algebraic reductive subgroup.\\
$a)$ If the representation of $G$
in $L^2(G/H)$ is tempered, then
the set of points in $G/H$ with amenable stabilizer in $H$ is dense.\\
$b)$ If the set of points 
in $G/H$ with abelian stabilizer in $\mathfrak h$ is dense, then
the representation of $G$
in $L^2(G/H)$ is tempered.
\end{theorem}

The proof of Theorem \ref{thXtemp2} leads us to
the list of all the spaces $G/H$ for which the representation of $G$ in $L^2(G/H)$ is non-tempered.

%

\subsection{Complex homogeneous spaces}
\label{seccomhom}

We assume in this section that $G$ and $H$ are complex Lie groups.
Since complex amenable reductive Lie groups are abelian,
 the following result is a particular case of
Theorem \ref{thXtemp2}.

\begin{theorem}
\label{exagxg2}
Suppose $G$ is a complex algebraic semisimple group
and $H$ a complex reductive subgroup.
Then 
$L^2(G/H)$ is tempered if and only if the set of points in $G/H$ 
with abelian stabilizer in $\mathfrak{h}$ is dense.
\end{theorem}

\begin{example}
\label{exagslsoc}
$L^2(SL(n,\mathbb{C})/SO(n,\mathbb{C}))$ is always tempered.\\
$L^2(SL(2m,\mathbb{C})/Sp(m,\mathbb{C}))$ is never tempered.\\
$L^2(SO(7,\mathbb{C})/G_2)$ is not tempered.
\end{example}

\begin{example}
\label{exaslnc}
Let $n=n_1+\cdots + n_r$ with $n_1\geq\cdots\geq n_r\geq 1$, $r\geq 2$.\\
$L^2(SL(n,\mathbb{C})/\prod SL(n_i,\mathbb{C}))$ is tempered iff
$2n_1\leq n+1$.\\
$L^2(SO(n,\mathbb{C})/\prod SO(n_i,\mathbb{C}))$ is tempered iff
$2n_1\leq n+2$.\\
$L^2(Sp(n,\mathbb{C})/\prod Sp(n_i,\mathbb{C}))$ is tempered iff
$r\geq 3$ and $2n_1\leq n$.
\end{example}

\subsection{Real homogeneous spaces}
\label{secsym}

Here are a few examples of application of our criterion.

\begin{example}
\label{exagxg}
Let $G_1$ be a real algebraic semisimple Lie group and 
$K_1$ a maximal compact subgroup.\\
$L^2(G_1\times G_1/\Delta(G_1))$ is always tempered.\\
$L^2(G_{1,\mathbb{C}}/G_1)$ is always tempered.\\
$L^2(G_{1,\mathbb{C}}/K_{1,\mathbb{C}})$ is tempered iff $G_1$ is quasisplit.
\end{example}

\begin{example}
\label{exafj} 
Let $G/H$ be a symmetric space i.e. $G$ is a real algebraic semisimple Lie group and $H$ is the set of fixed points
of an involution of $G$.
Write $\mathfrak{g}=\mathfrak{h} +  \mathfrak{q}$ for the $H$-invariant decomposition of $\mathfrak{g}$.
Let $G'$ be an algebraic semisimple Lie group with Lie algebra 
$\mathfrak{g}'=\mathfrak{h} + \sqrt{-1}\mathfrak{q}$. 

Then
$L^2(G/H)$ is almost $L^p$ iff $L^2(G'/H)$ is almost $L^p$.
\end{example}

\begin{example}
\label{exagslopq}
$L^2(SL(p+q,\mathbb{R})/SO(p,q))$ is always tempered.\\
$L^2(SL(2m,\mathbb{R})/Sp(m,\mathbb{R}))$ is never tempered.\\
$L^2(SL(m+n,\mathbb{R})/SL(m,\mathbb{R})\times SL(n,\mathbb{R}))$ is tempered iff $|m-n|\leq 1$.
\end{example} 

\begin{example}
\label{ex:soprod}
Let
$p_1 + \dots + p_r\le p$ and  $q_1 + \dots + q_r\le q$.\\
$L^2(SO(p,q)/\prod SO(p_i,q_i))$ is tempered iff
$\displaystyle 
2\max_{p_iq_i\neq 0} (p_i+q_i)\leq p+q+2$.
\end{example}

The homogeneous spaces in Examples~\ref{exaslnc} and \ref{ex:soprod}
are not symmetric spaces when $r \ge 3$.

\vspace{0.5cm}
\subsection*{Acknowledgments}

The authors are grateful to the Institut des Hautes \'Etudes Scientifiques (Bures-sur-Yvette) and to the University of Tokyo for its support through the GCOE program
for giving us opportunities to work together in very good conditions.
The second author was partially supported by Grant-in-Aid for Scientific Research 
(A) (25247006) JSPS.


{\small
\noindent
Y. \textsc{Benoist}\newline
CNRS-Universit\'e Paris-Sud, 91405 Orsay, France\newline
(e-mail: \texttt{yves.benoist@math.u-psud.fr})

\medskip
\noindent
T. \textsc{Kobayashi}\newline
Kavli IPMU (WPI) and Graduate School of Mathematical Sciences,
the University of Tokyo, Meguro, Komaba, 153-8914, Tokyo, Japan\newline
(e-mail: \texttt{toshi@ms.u-tokyo.ac.jp})
}

\end{document}